\documentclass[11pt,a4paper,twoside,reqno]{amsart}

\usepackage[T1]{fontenc}
\usepackage[utf8]{inputenc}
\usepackage{lmodern}
\usepackage{microtype}
\usepackage{amsmath,amssymb,amsthm,mathtools}
\usepackage[a4paper,margin=30mm]{geometry}
\usepackage{cite}
\usepackage[hidelinks,hypertexnames=false]{hyperref}
\hypersetup{
 pdftitle={Sharp Sobolev inequalities on minimal immersions via optimal transport},
 pdfauthor={Zhe-Feng Xu and Zhuo-Nan Zhu},
 pdfsubject={Sharp Sobolev, Gagliardo-Nirenberg and dual transport inequalities},
 pdfkeywords={Sobolev inequality, minimal immersion, optimal transport}
}

\numberwithin{equation}{section}
\newtheorem{theorem}{Theorem}[section]
\newtheorem{proposition}[theorem]{Proposition}
\newtheorem{corollary}[theorem]{Corollary}
\newtheorem{lemma}[theorem]{Lemma}
\theoremstyle{remark}
\newtheorem{remark}[theorem]{Remark}

\newcommand{\R}{\mathbb R}
\newcommand{\dd}{\,\mathrm d}
\newcommand{\vol}{\mathrm{vol}}
\newcommand{\AT}{\operatorname{AT}}
\newcommand{\II}{\mathrm{II}}
\newcommand{\tr}{\operatorname{tr}}
\newcommand{\supp}{\operatorname{supp}}
\newcommand{\ac}{\mathrm{ac}}
\newcommand{\JE}{J_{\mathrm E}}
\newcommand{\ind}{\mathbf 1}
\newcommand{\Euc}{\mathrm{Euc}}

\makeatletter

\renewcommand{\@maketitle}{%
  \normalfont\normalsize
 
  \begingroup
    \let\@date\@empty
    \@adminfootnotes
  \endgroup
  \begingroup
    \centering
    {\LARGE\mdseries\@title\par}
    \vskip 1.6em
    {\Large
      \def\and{\unskip,\space}%
      \authors\par}
    \vskip 1em
    {\Large\@date\par}
  \endgroup
  \vskip 3em
  \@setabstract
  \vskip 1.2em
}

\renewenvironment{abstract}{%
  \global\setbox\abstractbox=\vtop\bgroup
    \normalfont\normalsize
    \hsize=\textwidth
    \linewidth=\hsize
    \parindent=0pt
    \noindent\textbf{Abstract}\enspace\ignorespaces
}{%
    \par
  \egroup
}

\renewcommand{\@setabstract}{%
  \ifvoid\abstractbox\else
    \box\abstractbox
  \fi
}

\makeatother

\title[Sharp Sobolev inequalities on minimal immersions]
{Sharp Sobolev inequalities on minimal immersions\\
via optimal transport}

\author[Zhe-Feng Xu]{Zhe-Feng Xu\textsuperscript{1,\,2}}
\author[Zhuo-Nan Zhu]{Zhuo-Nan Zhu\textsuperscript{1}}

\date{\today}

\begin{document}

\begin{abstract}
We prove that, for $2\leq p<n$, the optimal $L^p$-Sobolev constant
on every complete $n$-dimensional Euclidean minimal immersion
without boundary is the Euclidean Aubin--Talenti constant.
The proof combines optimal transport to a suitably chosen target
with a two-kernel comparison linking the tangential transport moment to deficits in the normal-fibre and Jacobian estimates.
\end{abstract}

\maketitle

\footnotetext[1]{School of Mathematical Sciences,
University of Science and Technology of China,
230026 Hefei, China.
xzf1998@mail.ustc.edu.cn, zhuonanzhu@mail.ustc.edu.cn}

\footnotetext[2]{SISSA, 34136 Trieste, Italy.
zxu@sissa.it}

\setcounter{footnote}{2}


\medskip


\bigskip

\section{Introduction}

Let $n\geq2$ and $1<p<n$. We write $q=p'=p/(p-1)$ and
$p^*=np/(n-p)$ for the conjugate and critical Sobolev exponents,
respectively. The classical Euclidean Sobolev inequality states that
\begin{equation*}
 \left(\int_{\R^n}|f|^{p^*}\dd x\right)^{1/p^*}
 \leq \AT(n,p)
 \left(\int_{\R^n}|\nabla f|^p\dd x\right)^{1/p},
 \qquad f\in C_c^\infty(\R^n),
\end{equation*}
where the best constant, determined by Aubin~\cite{Aubin} and
Talenti~\cite{Talenti}, is
\begin{equation}\label{AT}
 \AT(n,p)=\pi^{-1/2}n^{-1/p}
 \left(\frac{p-1}{n-p}\right)^{1/q}
 \left[
 \frac{\Gamma(n/2+1)\Gamma(n)}
 {\Gamma(n/p)\Gamma(n/q+1)}
 \right]^{1/n}.
\end{equation}
The endpoint $p=1$ is the functional form of the Euclidean
isoperimetric inequality, with constant $(n\omega_n^{1/n})^{-1}$,
where $\omega_n=|B_1^n|$. In this paper, we study whether the constant in
\eqref{AT} remains optimal for the Sobolev inequality on minimal
submanifolds of Euclidean space.

Early Sobolev inequalities on minimal hypersurfaces were obtained
by Miranda~\cite{MirandaSobolev} and by Bombieri, De Giorgi and
Miranda~\cite{BDGM}.  Allard~\cite{Allard}
and Michael--Simon~\cite{MichaelSimon} established inequalities
on general Euclidean submanifolds involving their mean curvature.
More precisely, let $m\geq1$ and let $\Sigma\subset\R^{n+m}$ be a complete
$n$-dimensional submanifold without boundary. We denote its induced
metric by $g$, its Riemannian volume by $\vol_\Sigma$, and its
unnormalized mean curvature vector by $H$. We write $\nabla^\Sigma$
and $\Delta_\Sigma=\operatorname{div}_{\Sigma}\nabla^\Sigma$ for the
gradient and Laplace--Beltrami operator. Then
\begin{equation}\label{MSSI}
    \left(\int_\Sigma |f|^{\frac{n}{n-1}}\dd\vol_\Sigma\right)^{(n-1)/n}
 \leq \frac{4^{n+1}}{\omega_n^{1/n}}
 \int_\Sigma\bigl(|\nabla^\Sigma f|+|H||f|\bigr)\dd\vol_\Sigma,
 \qquad f\in C_c^\infty(\Sigma).
\end{equation}
For a direct proof in codimension one and related universal
Hardy--Sobolev inequalities, see Cabr\'e and
Miraglio~\cite{CabreMiraglio}. Hoffman and Spruck~\cite{HoffmanSpruck} extended
mean-curvature Sobolev inequalities to immersed submanifolds
of Riemannian manifolds under suitable geometric and
support assumptions. Castillon~\cite{Castillon} used optimal transport to establish
sharp weighted Sobolev inequalities and a nonsharp unweighted
Michael--Simon type inequality on Euclidean submanifolds.
For minimal submanifolds, the mean curvature term in \eqref{MSSI} vanishes. A natural question is
whether its constant can be replaced by the Euclidean best constant.

Brendle~\cite{BrendleIsoperimetric} proved a Michael--Simon--Sobolev
inequality in arbitrary codimension by the
Alexandrov--Bakelman--Pucci (ABP) method, with a codimension-dependent
constant that is sharp in codimensions one and two.
Brendle and Eichmair~\cite{BETransport} gave an optimal transport
proof of its isoperimetric form; see also~\cite{BENotices}.
For minimal submanifolds in these codimensions, the coarea
formula and decreasing rearrangement yield the sharp
$L^p$-Sobolev constant $\AT(n,p)$ for every $1<p<n$;
see~\cite[Theorem~5.8]{BrendleSurvey} for the hypersurface
case and~\cite{AldrigoBalogh} for the rearrangement approach. For a refinement of Brendle's inequality in codimension at least
three, see Sun~\cite{Sun}. 
In arbitrary codimension, Balogh, Krist\'aly and
Mester~\cite{BKM} used a normal-bundle formulation of optimal
transport to obtain $L^p$-Sobolev inequalities on minimal
submanifolds with a codimension-independent constant $S(n,p)$
for $2\leq p<n$, satisfying
\[
 S(n,p)>\AT(n,p),\qquad
 \lim_{n\to\infty}\frac{S(n,p)}{\AT(n,p)}=1
\]
for every fixed $p\geq2$. 
Their estimate for $1<p<2$ remains codimension-dependent.
The loss of sharpness comes from the estimate of the tangential
transport moment; see~\cite[Remark~2.1]{BKM}.
In the anisotropic setting, Han and Xu~\cite{HanXu} obtained
sharp anisotropic
$L^p$-Sobolev inequalities for minimal submanifolds in codimensions
three and four, for $1<p<n$ and ellipsoidal anisotropies.
Their proof combines an inverse optimal transport construction
for the ABP map, a sharp normal-fibre estimate, and convex
symmetrization. 

Related optimal transport methods have also been used in the
logarithmic Sobolev setting. Cordero-Erausquin
~\cite{CorderoGaussian} gave a direct optimal transport proof of
the Gaussian logarithmic Sobolev inequality. Balogh, Don and
Krist\'aly~\cite{BDKLog} established sharp weighted $L^p$-logarithmic Sobolev inequalities for log-concave homogeneous
weights on open convex cones and characterized the equality
cases. In the submanifold setting, Brendle~\cite{BrendleLog}
established the sharp $L^2$-logarithmic Sobolev inequality in
arbitrary codimension, while Balogh and Krist\'aly~\cite{BK}
subsequently gave an optimal transport treatment and established rigidity. 

Motivated by these results, a natural question arises: can one obtain
the Euclidean sharp $L^p$-Sobolev inequality in arbitrary codimension?
The purpose of this paper is to answer this question for minimal
immersions when $2\leq p<n$. Our main theorem is as follows:
\begin{theorem}\label{Sob}
Let $n\geq3$, $m\geq1$, and $2\leq p<n$. Suppose that $X:(\Sigma,g)\to\R^{n+m}$ is a complete $n$-dimensional minimal isometric immersion without boundary. Then
\begin{equation}\label{SobI}
 \left(\int_\Sigma |f|^{p^*}\dd\vol_\Sigma\right)^{1/p^*}
 \leq \AT(n,p)
 \left(\int_\Sigma |\nabla^\Sigma f|^p\dd\vol_\Sigma\right)^{1/p},
 \qquad f\in C_c^\infty(\Sigma).
\end{equation}
The constant $\AT(n,p)$ is optimal on each fixed
$\Sigma$. Moreover, no nonzero function in $C_c^\infty(\Sigma)$ attains equality.
\end{theorem}

Theorem~\ref{Sob} also determines the best constants in
two related classes of inequalities. We first consider
Gagliardo--Nirenberg inequalities. The sharp Euclidean constants in the inequalities below were
determined by Del Pino and Dolbeault~\cite{DPD}; 
an optimal transport proof was given in~\cite{CNV}. On minimal
submanifolds, rearrangement gives the same constants whenever the
sharp isoperimetric inequality is available; see~\cite{AldrigoBalogh}. Applying Theorem~\ref{Sob} to products with Euclidean spaces, we obtain the following sharp Gagliardo--Nirenberg inequalities in arbitrary codimension.

\begin{corollary}\label{GN}
Let $n\geq2$ and $m\geq1$. Suppose that
$X:(\Sigma,g)\to\R^{n+m}$ is a complete $n$-dimensional
minimal isometric immersion without boundary.
For each integer $k\geq1$, put
\begin{equation}\label{GNExp}
 d=n+k,\qquad t=\frac{2n+k}{2n+k-4},\qquad
 \theta=\frac{n}{t(d-2)}.
\end{equation}
Then
\begin{equation}\label{GNI}
 \|f\|_{L^{2t}(\Sigma)}
 \leq G_{n,t}\|\nabla^\Sigma f\|_{L^2(\Sigma)}^\theta
                 \|f\|_{L^{t+1}(\Sigma)}^{1-\theta},
 \qquad f\in C_c^\infty(\Sigma),
\end{equation}
where $G_{n,t}$ is the Euclidean best constant. With
$a=1/(t-1)$ and $b=2a+1$, it is given by
\begin{equation*}
 G_{n,t}
 =\left(\frac{b-n/2}{b}\right)^{1/(2t)}
  \left(\frac{b}{2na^2}\right)^{\theta/2}
  \left[\frac{\Gamma(b)}{\pi^{n/2}\Gamma(b-n/2)}\right]^{\theta/n}.
\end{equation*}
The constant is optimal on each fixed $\Sigma$.
\end{corollary}

The proof applies Theorem~\ref{Sob} with $p=2$ to the
product immersion $(x,y)\mapsto(X(x),y)$ of $\Sigma\times\R^k$.
Since $n+k\geq3$, this also covers $n=2$. The values of $t$ in \eqref{GNExp} form a discrete sequence decreasing to $1$ as $k\to\infty$. The proof follows the Euclidean dimension reduction argument in \cite[Section 3]{Nguyen}.

The second consequence concerns Sobolev duality and the Beckmann
problem. For $q>1$, define
\begin{equation}\label{Beck}
 \mathcal B_q(h)=\inf\left\{
 \frac1q\int_\Sigma|Z|^q\dd\vol_\Sigma :
 Z\in L^q(T\Sigma),\ -\operatorname{div}_{\Sigma} Z=h\right\}.
\end{equation}
The divergence constraint is understood in the sense of distributions, and no additional boundary condition at infinity is imposed. We refer to this as the
homogeneous Beckmann problem with power cost; see
\cite{Beckmann,BrascoPetrache} for continuous transport formulations.

\begin{corollary}\label{Dual}
Let $n\geq3$, $m\geq1$, and $2\leq p<n$, and set
$r=nq/(n+q)=(p^*)'$.
Suppose that $X:(\Sigma,g)\to\R^{n+m}$ is a complete
$n$-dimensional minimal isometric immersion without boundary.
Then, for every $h\in L^r(\Sigma)$, the infimum in
\eqref{Beck} is attained by a unique vector field
$Z_h\in L^q(T\Sigma)$, and
\begin{equation}\label{BI}
 \mathcal B_q(h)\leq\frac{\AT(n,p)^q}{q}
 \|h\|_{L^r(\Sigma)}^q.
\end{equation}
The constant is optimal on each fixed $\Sigma$, even when
$h=\rho-\eta$, where $\eta\in C_c^\infty(\Sigma)$ is a fixed
probability density and $\rho$ varies over such densities. 

For $p=2$, the minimal Green kernel $G_\Sigma$ of $-\Delta_\Sigma$ exists,
and for all $f,h\in L^{2n/(n+2)}(\Sigma)$,
\begin{equation}\label{HLS}
 \left|\int_\Sigma\int_\Sigma G_\Sigma(x,y)f(x)h(y)
            \dd\vol_\Sigma(x)\dd\vol_\Sigma(y)\right|
 \leq\AT(n,2)^2\|f\|_{L^{2n/(n+2)}(\Sigma)}
                  \|h\|_{L^{2n/(n+2)}(\Sigma)}.
\end{equation}
The constant is optimal on each fixed $\Sigma$.
\end{corollary}

On an affine $n$-plane,
\[
 G_{\R^n}(x,y)=\frac{|x-y|^{2-n}}{n(n-2)\omega_n}.
\]
Thus \eqref{HLS} is the Newtonian case of the sharp
Hardy--Littlewood--Sobolev inequality~\cite{LiebHLS}.
Here $G_\Sigma$ is the intrinsic Green kernel, not a power of the
ambient distance. Both statements in Corollary~\ref{Dual}
follow by duality from Theorem~\ref{Sob}.

The proof of Theorem~\ref{Sob} follows the optimal transport argument of
Cordero-Erausquin, Nazaret and Villani~\cite{CNV}. The transport
on submanifolds is described by an integral Monge--Amp\`ere
identity; we use the results of Balogh and Krist\'aly~\cite{BK}
and Wang~\cite{Wang}. For a general treatment of transport between
unequal dimensions, see McCann and Pass~\cite{McCannPass}.

For embedded submanifolds, our proof refines the optimal transport argument of Balogh,
Krist\'aly and Mester~\cite{BKM} by improving the nonsharp estimate
discussed in their Remark~2.1. The key point is to relate the
tangential transport moment to the deficits in the normal-fibre
and Jacobian estimates, rather than estimate it independently
of these deficits. To be more precise, we construct two positive radial kernels whose marginals on every
$n$-dimensional linear subspace are
\[
g_{n,p}(z)=
\frac{(1+|z|^q)^{-n}}
{\displaystyle\int_{\mathbb R^n}(1+|w|^q)^{-n}\,\mathrm dw}
\qquad\text{and}\qquad
\frac{1+|z|^q}{n-1}g_{n,p}(z),
\]
respectively. The construction of these kernels relies on a positive Laplace representation and requires $p\ge2$. Comparing the two kernels on each restricted normal fibre relates the tangential transport moment to the loss in the
transport estimate. Together with a scalar concavity inequality,
this relation yields the sharp constant $\AT(n,p)$ in the
embedded case, without requiring the tangential moment to equal
its Euclidean value. The extension to arbitrary minimal immersions
then follows by approximation.

The paper is organized as follows. In
Section~\ref{OT}, we recall the compact-target
transport construction, construct the two kernels, and prove the
normal-fibre estimate. We prove Theorem~\ref{Sob} in
Section~\ref{Proof} and the two corollaries in
Section~\ref{App}.

\section{Optimal transport and normal-fibre estimates}
\label{OT}

\subsection{Optimal transport on submanifolds}

Let $\Sigma\subset\R^N$ be a smooth embedded $n$-dimensional
submanifold, with $N>n\geq2$. For compactly supported probability
measures $\mu=F\dd\vol_\Sigma$ and $\nu=G\dd y$, consider the optimal transport problem with quadratic cost:
\begin{equation}\label{KP}
 \inf_{\pi\in\Pi(\mu,\nu)}
 \int_{\Sigma\times\R^N}\frac12|x-y|^2\dd\pi(x,y),
\end{equation}
where $\Pi(\mu,\nu)$ denotes the set of probability measures on
$\Sigma\times\R^N$ with marginals $\mu$ and $\nu$.
The source measure $\mu$ is absolutely continuous with respect to
$\vol_\Sigma$, but singular with respect to $N$-dimensional
Lebesgue measure. The following proposition represents an optimal
transport plan using the normal bundle of $\Sigma$.

We write
$T_x^\perp\Sigma$ for the normal space, $\II$ for the second
fundamental form, and $H=\tr_\Sigma\II$ for the unnormalized mean
curvature vector. For a locally semiconvex function $u$, let
$D_\Sigma^2u$ denote its Alexandrov Hessian and
$\Delta^\Sigma_{\ac}u=\tr D_\Sigma^2u$ the density of the absolutely
continuous part of its distributional Laplacian
$\Delta^\Sigma_{\mathcal D}u$. The latter is a locally finite signed
Radon measure. Denote by $B_L^N$ the open ball in $\R^N$ of
radius $L>0$ centred at the origin.

\begin{proposition}\label{OTProp}
Let $N>n\geq2$, and let $\Sigma\subset\R^N$ be a complete
$n$-dimensional embedded submanifold without boundary. Let
$\mu=F\dd\vol_\Sigma$ and $\nu=G\dd y$ be
compactly supported probability
measures, and assume that $G=0$
outside $B_L^N$ for some $L>0$. Then there exist a
finite locally semiconvex function $u:\Sigma\to\R$, a measurable
set $A\subset T^\perp\Sigma$ lying over $\{F>0\}$, and an injective map
\[
 \Phi:A\to B_L^N,\qquad
 \Phi(x,v)=\nabla^\Sigma u(x)+v,
\]
such that the projection of $A$ has full $F\dd\vol_\Sigma$-measure
and $\Phi(A)$ has full $G\dd y$-measure. The function $u$ is twice differentiable in the Alexandrov sense
on the projection of $A$, and
\begin{equation}\label{MA}
 F(x)=\int_{A_x}G(\Phi(x,v))\det Q_{x,v}\dd v,
 \qquad A_x=A\cap T_x^\perp\Sigma,
\end{equation}
for $F\dd\vol_\Sigma$-almost every $x$, where $Q_{x,v}=D_\Sigma^2u(x)-\langle\II(x),v\rangle\geq0$. For $\dd\vol_\Sigma(x)\dd v$-almost every $(x,v)\in A$, the approximate Jacobian of $\Phi$ equals $\det Q_{x,v}$, and 
\begin{equation}\label{Det}
 0\leq\det Q_{x,v}
 \leq\left(\frac{\Delta^\Sigma_{\ac}u(x)-\langle H(x),v\rangle}{n}\right)^n.
\end{equation}
Set
\[
 \dd\vartheta(x,v)=\ind_A(x,v)G(\Phi(x,v))\det Q_{x,v}
                   \dd\vol_\Sigma(x)\dd v.
\]
Then $\vartheta$ is a probability measure and
\begin{equation}\label{Plan}
 \pi=\bigl((x,v)\mapsto(x,\Phi(x,v))\bigr)_\#\vartheta
\end{equation}
is an optimal coupling for \eqref{KP}.
Moreover, $|\nabla^\Sigma u|\leq L$ almost everywhere and
$\Delta^\Sigma_{\ac}u\leq\Delta^\Sigma_{\mathcal D}u$ in the sense of
distributions.
\end{proposition}

\begin{proof}
We recall the contact construction from~\cite[Theorem~2.1 and its
proof]{BK}; see also~\cite{Wang}. We give the compact dual
construction to ensure that the potential is finite on all of
$\Sigma$.

Let $S=\supp\mu$. Since both supports are compact, Kantorovich
duality provides continuous functions $a_0$ on $S$ and $b_0$ on
$\overline B_L^N$ such that $a_0(x)+b_0(y)\geq\langle x,y\rangle$,
with equality on an optimal coupling; see~\cite{Villani}.
Set
\[
 \psi(y)=\sup_{x\in S}\{\langle x,y\rangle-a_0(x)\},\qquad
 w(z)=\sup_{y\in\overline B_L^N}
             \{\langle z,y\rangle-\psi(y)\}.
\]
The function $\psi$ is finite, convex and Lipschitz on $\R^N$.
Also, $\psi\leq b_0$ on $\overline B_L^N$ and $w\leq a_0$ on
$S$, and both inequalities are equalities on transport pairs.
Since $\nu$ is absolutely continuous, the Brenier--McCann
theorem~\cite{Brenier,McCann} shows that the reverse optimal
coupling is induced by $T=\nabla\psi$, with $T_\#\nu=\mu$.

Similarly,  the function $w$ is finite, convex and $L$-Lipschitz. On $\Sigma$,
each function $x\mapsto\langle x,y\rangle-\psi(y)$ has Hessian
$\langle\II,y\rangle$, uniformly bounded from below on compact
sets for $|y|\leq L$. Thus $u=w|_\Sigma$ is locally semiconvex
and $|\nabla^\Sigma u|\leq L$ almost everywhere.
For $\nu$-almost every $y$, the point $x=T(y)$ lies in $\{F>0\}$,
$u$ is twice differentiable at $x$, and $u(x)+\psi(y)=\langle x,y\rangle$.
The function $x'\mapsto u(x')-\langle x',y\rangle$ therefore has a minimum
at $x$, and its first and second derivatives give 
\[
 y=\nabla^\Sigma u(x)+v,\qquad v\in T_x^\perp\Sigma,\qquad
 D_\Sigma^2u(x)-\langle\II(x),v\rangle\geq0.
\]
To apply the area formula, choose a Borel set $E\subset\{F>0\}$
of full $\mu$-measure on which $u$ is twice differentiable in the
Alexandrov sense and $\nabla^\Sigma u$ admits a countable Lipschitz
decomposition in local coordinates;  see~\cite[Sections~3.1--3.2]{Federer}. Since $T_\#\nu=\mu$, we have $\nu(T^{-1}(E))=1$.
We then let $A\subset T^\perp\Sigma$ be the measurable contact
set associated with $T$, constructed as in
\cite[Theorem~2.1]{BK}; see also \cite[Theorem~3.3]{Wang}.
By construction, $\Phi$ is injective and $\Phi(A)$ has full
$\nu$-measure. Moreover, using smooth local normal frames and
restricting to bounded pieces, we see that $\Phi$ is Lipschitz
on countably many measurable pieces.

On these Lipschitz pieces, the approximate Jacobian of $\Phi$ equals $\det Q_{x,v}$ almost everywhere, as in~\cite{BK,Wang}, and the
arithmetic--geometric mean inequality gives \eqref{Det}. The area formula, together with $T_\#\nu=\mu$, then yields \eqref{MA} and the two marginals in \eqref{Plan}; see also~\cite{Federer}. We may further restrict $A$ to
the points over the full $\mu$-measure set where \eqref{MA}
holds. Since $T_\#\nu=\mu$, this restriction again preserves
full $\nu$-measure of the image.

Finally, $u(x)+\psi(y)\geq\langle x,y\rangle$ for $|y|\leq L$,
with equality when $y=\Phi(x,v)$ for $\vartheta$-almost every
$(x,v)\in A$. This proves the optimality of $\pi$. The last assertion follows from the local semiconvexity of $u$;
see~\cite[Theorem~2.1(v)]{BK}.
\end{proof}

In the following, we use the measure $\dd\vol_\Sigma(x)\dd v$ on the normal bundle,
where $\dd v$ is Lebesgue measure on $T_x^\perp\Sigma$. The area
formula on the countably Lipschitz pieces of $\Phi$ yields 
\begin{equation*}
 \int_\Sigma\int_{A_x}\ell(\Phi(x,v))\det Q_{x,v}\dd v\dd\vol_\Sigma(x)
 =\int_{\Phi(A)}\ell(y)\dd y
\end{equation*}
for every nonnegative measurable function $\ell$; see~\cite{Federer,Wang}.
The set $A$ has already been restricted to points over $\{F>0\}$
where \eqref{MA} holds. In our application, $G$ is
positive almost everywhere in $B_L^N$, so $B_L^N\setminus\Phi(A)$
is Lebesgue null. Hence
\begin{equation}\label{AreaB}
 \int_\Sigma\int_{A_x}\ell(\Phi(x,v))\det Q_{x,v}\dd v\dd\vol_\Sigma(x)
 =\int_{B_L^N}\ell(y)\dd y.
\end{equation}
In particular, the change of variables formula also applies to
the auxiliary kernel $K$ introduced below. Since
$\nabla^\Sigma u(x)\in T_x\Sigma$ and $v\in T_x^\perp\Sigma$,
we also have
\begin{equation}\label{Ball}
 |\nabla^\Sigma u(x)|^2+|v|^2
 =|\Phi(x,v)|^2\leq L^2,\qquad \forall\,(x,v)\in A.
\end{equation}

\subsection{Two kernels with prescribed marginals}

Fix $n\geq3$, $2\leq p<n$, and $N>n$, and write $m=N-n$.
Set
\begin{gather}
c_{n,p}=\int_{\R^n}(1+|z|^q)^{-n}\dd z
=\omega_n\frac{\Gamma(n/q+1)\Gamma(n/p)}{\Gamma(n)},
\label{c}\\
\JE=\frac{n(p-1)}{n-p},
\qquad 
\beta=\frac{p(n-1)}{n-p}=1+\JE.
\notag
\end{gather}
Then, the Aubin--Talenti constant can be written as
\begin{equation}\label{ATJ}
 \AT(n,p)=\frac1n c_{n,p}^{-1/n}\JE^{1/q}.
\end{equation}
Let $\alpha=q/2\in(1/2,1]$.
Since $s\mapsto s^\alpha$ is a Bernstein function,
$s\mapsto e^{-\lambda s^\alpha}$ is completely monotone
for every $\lambda>0$; see~\cite[Chapter~3]{Bernstein}.
Bernstein's theorem~\cite[Theorem~1.4]{Bernstein}
therefore gives a probability measure $\sigma_{\alpha,\lambda}$
on $(0,\infty)$ such that
\begin{equation*}
 \int_0^\infty e^{-st}\dd\sigma_{\alpha,\lambda}(t)
 =e^{-\lambda s^\alpha},\qquad \forall\,s\geq0.
\end{equation*}
These measures can be chosen measurably in $\lambda$ by dilating
$\sigma_{\alpha,1}$ by the factor $\lambda^{1/\alpha}$.
For $\alpha=1$, one has $\sigma_{1,\lambda}=\delta_\lambda$.
Define
\begin{equation*}
 \mathcal L_{N,\lambda}(y)
 =\pi^{-m/2}\int_0^\infty t^{m/2}e^{-t|y|^2}
       \dd\sigma_{\alpha,\lambda}(t),\qquad \forall\,y\in\R^N.
\end{equation*}
For every $n$-dimensional linear subspace $V\subset\R^N$ and
$z\in V$, integration in the normal variables and using Fubini's theorem yield 
\begin{equation}\label{LapM}
 \int_{V^\perp}\mathcal L_{N,\lambda}(z+v)\dd v
 =e^{-\lambda|z|^q}.
\end{equation}

We now introduce the following two radial kernels:
\begin{align*}
 \mathcal {G}_{N,p}(y)
 &=\frac1{c_{n,p}\Gamma(n)}\int_0^\infty
       \lambda^{n-1}e^{-\lambda}\mathcal L_{N,\lambda}(y)\dd\lambda,\\
 \mathcal {K}_{N,p}(y)
 &=\frac1{c_{n,p}\Gamma(n)}\int_0^\infty
       \lambda^{n-2}e^{-\lambda}\mathcal L_{N,\lambda}(y)\dd\lambda.
\end{align*}
Both kernels are positive and finite away from the origin. Indeed, since \(\sigma_{\alpha,\lambda}\) is a probability measure, for \(y\neq0\),
\[
0<\mathcal L_{N,\lambda}(y)
\leq \pi^{-m/2}\sup_{t>0}t^{m/2}e^{-t|y|^2}<\infty,
\]
uniformly in \(\lambda>0\). The defining integrals therefore converge,
since $\lambda^{n-1}e^{-\lambda}$ and
$\lambda^{n-2}e^{-\lambda}$ are integrable on $(0,\infty)$.
We set both kernels equal to zero at the origin. 

For the remainder of the paper, we write $\mathcal {G}=\mathcal {G}_{N,p}$ and $\mathcal {K}=\mathcal {K}_{N,p}$,
and set $a_+=\max\{a,0\}$ for $a\in\R$. We also set
\[
 g_{n,p}(z)=c_{n,p}^{-1}(1+|z|^q)^{-n},\qquad z\in\R^n,
\]
which is the normalized $p^*$-power of an Aubin--Talenti extremal. Since $g_{n,p}$ is radial, we use the same notation for the
corresponding density on any $n$-dimensional Euclidean subspace. 

\begin{proposition}
For every $n$-dimensional linear subspace $V\subset\R^N$ and $z\in V$,
\begin{align}
 \int_{V^\perp}\mathcal{G}(z+v)\dd v
 &=g_{n,p}(z),
       \label{Gm}\\
 \int_{V^\perp}\mathcal{K}(z+v)\dd v
 &=\frac{1+|z|^q}{n-1}g_{n,p}(z),
       \label{Km}
\end{align}
where $V^\perp$ denotes the orthogonal complement of $V$ in $\R^N$
and $\dd v$ is Lebesgue measure on $V^\perp$. In particular, $\mathcal{G},\mathcal{K}\in L^1(\R^N)$ and
\begin{equation}\label{Mass}
 \int_{\R^N}\mathcal{G}(y)\dd y=1,\qquad
 \int_{\R^N}\mathcal{K}(y)\dd y=\frac{1+\JE}{n-1}.
\end{equation}
Moreover, $\int_{\R^N}|y|^q\mathcal{G}(y)\dd y<\infty.$
\end{proposition}

\begin{proof}
By Tonelli's theorem and \eqref{LapM}, together with 
\[
\int_0^\infty\lambda^{a-1}e^{-b\lambda}\dd\lambda
=\Gamma(a)b^{-a},
\qquad a,b>0,
\]
we obtain \eqref{Gm} and \eqref{Km}. Moreover, note that 
\begin{equation}\label{J}
 \int_{\R^n}|z|^qg_{n,p}(z)\dd z
 =\frac{\mathrm{B}(n/q+1,n/p-1)}{\mathrm{B}(n/q,n/p)}
 =\frac{n(p-1)}{n-p}=\JE, 
\end{equation}
where $\mathrm{B}$ denotes the beta function. Integrating the marginal identities over $V$ and applying
Tonelli's theorem gives \eqref{Mass}. To prove the moment bound, fix $i=1,\ldots,N$ and choose an $n$-dimensional linear subspace
$V_i\subset\R^N$ containing the coordinate vector $e_i$. If
$y=z+v$ with $z\in V_i$ and $v\in V_i^\perp$, then $y_i=z_i$.
Hence, by Tonelli's theorem, \eqref{Gm}, and \eqref{J}, we have  
\[
\int_{\R^N}|y_i|^q\mathcal{G}(y)\dd y
=\int_{V_i}|z_i|^qg_{n,p}(z)\dd z
\leq\int_{V_i}|z|^qg_{n,p}(z)\dd z
=\JE. 
\]
Since $q\leq2$, $|y|^q\leq\sum_{i=1}^N|y_i|^q$, and 
therefore $\mathcal{G}$ has a finite $q$-moment. 
\end{proof}

\begin{remark}
For $1<p<2$, we have $\alpha=q/2>1$. In this case, setting
$\varphi(s)=(1+s^\alpha)^{-n}$, we obtain
\[
\varphi''(s)
=n\alpha s^{\alpha-2}(1+s^\alpha)^{-n-2}
\bigl((n\alpha+1)s^\alpha-(\alpha-1)\bigr),
\]
which is negative for sufficiently small $s>0$.
It follows that $\varphi$ does not admit a positive Laplace
representation, and hence the construction above requires $p\geq2$.
\end{remark}

\subsection{Fibrewise comparison}

The following estimate relates the two kernels on a restricted
normal fibre. 

\begin{lemma}\label{Fib}
Let $V\subset\R^N$ be an $n$-dimensional linear subspace, let $z\in V$,
and let $h:V^\perp\to[0,1]$ be measurable. Set
\[
 r=\frac1{g_{n,p}(z)}\int_{V^\perp}h(v)\mathcal{G}(z+v)\dd v.
\]
If $r>0$, then
\begin{equation}\label{hKG}
 \frac{\displaystyle\int_{V^\perp}h(v)\mathcal{K}(z+v)\dd v}
      {\displaystyle\int_{V^\perp}h(v)\mathcal{G}(z+v)\dd v}
 \leq (1+|z|^q)\left(\frac n{n-1}r^{-1/n}-1\right).
\end{equation}
\end{lemma}

\begin{proof}
For $s\geq0$ and $y\ne0$, since $(1-s\lambda)_+\leq e^{-s\lambda}$, it holds that 
\begin{align*}
 (\mathcal{K}(y)-s\mathcal{G}(y))_+
 &\leq\frac1{c_{n,p}\Gamma(n)}\int_0^\infty
    \lambda^{n-2}e^{-\lambda}(1-s\lambda)_+
    \mathcal L_{N,\lambda}(y)\dd\lambda\\
 &\leq\frac1{c_{n,p}\Gamma(n)}\int_0^\infty
    \lambda^{n-2}e^{-(1+s)\lambda}
    \mathcal L_{N,\lambda}(y)\dd\lambda.
\end{align*}
Integrating over $z+V^\perp$ and using
\eqref{LapM}, we obtain
\begin{equation*}
 \int_{V^\perp}(\mathcal{K}(z+v)-s\mathcal{G}(z+v))_+\dd v
 \leq\frac{(1+|z|^q+s)^{1-n}}{c_{n,p}(n-1)}.
\end{equation*}
Since $0\leq h\leq1$ and $g_{n,p}(z)=c_{n,p}^{-1}(1+|z|^q)^{-n}$, we have  
\[
 \int_{V^\perp}h(v)\mathcal{K}(z+v)\dd v
 \leq s\,g_{n,p}(z)r+\frac{(1+|z|^q+s)^{1-n}}{c_{n,p}(n-1)}.
\]
Choosing $s=(1+|z|^q)(r^{-1/n}-1)\geq0$ to minimize the right-hand side, we obtain
\[
 \int_{V^\perp}h(v)\mathcal{K}(z+v)\dd v
 \leq g_{n,p}(z)(1+|z|^q)\left(\frac n{n-1}r^{1-1/n}-r\right).
\]
Dividing by $g_{n,p}(z)r$ yields \eqref{hKG}. When $h\equiv1$, we have $r=1$, and \eqref{Km} shows that
equality holds in \eqref{hKG}. 
\end{proof}

\begin{lemma}\label{Sc}
Let $1<p<n$ and $q=p/(p-1)$. If $J,\delta\geq0$
satisfy $\JE-J\leq n\delta$, then
\begin{equation*}
 \frac{J^{1/q}}{1+J+\delta}
 \leq\frac{\JE^{1/q}}{1+\JE}.
\end{equation*}
\end{lemma}

\begin{proof}
	By concavity of $t\mapsto t^{1/q}$ and the fact
	$(1+\JE)/(q\JE)-1=-1/n$, we obtain 
	\[
	\frac{1+\JE}{\JE^{1/q}}J^{1/q}
	\leq 1+J+\frac{\JE-J}{n}
	\leq 1+J+\delta.
	\]
	Dividing by $1+J+\delta$ proves the result.
\end{proof} 

\section{Proof of the sharp Sobolev inequality}
\label{Proof}

\subsection{The truncated transport estimate}

We retain the notation of Section~\ref{OT}.
Fix $N>n$. For $L>0$,
let
\begin{equation}\label{ZLML}
 Z_L=\int_{B_L^N}\mathcal{G}(y)\dd y,\qquad
 M_L=\int_{B_L^N}\mathcal{K}(y)\dd y,
\end{equation}
and
\begin{equation}\label{Err}
 \eta_L=\frac1n\left[1+\JE-(n-1)\frac{M_L}{Z_L}\right]_+.
\end{equation}
These quantities depend only on $n,p,N,L$. Since $\mathcal{G}>0$ almost
everywhere, $0<Z_L\leq1$. By \eqref{Mass} and monotone
convergence,
\begin{equation}\label{Lim}
 Z_L\longrightarrow1,\qquad
 M_L\longrightarrow\frac{1+\JE}{n-1},\qquad
 \eta_L\longrightarrow0, 
 \qquad\text{as }L\to\infty.
\end{equation}

\begin{proposition}\label{Tr}
Let $n\geq3$, $2\leq p<n$, and $N>n$. Let
$\Sigma\subset\R^N$ be a complete $n$-dimensional embedded submanifold
without boundary. Then, for every $L>0$ and every nonnegative $f\in C_c^\infty(\Sigma)$
with $\int_\Sigma f^{p^*}\dd\vol_\Sigma=1$,
\begin{equation}\label{TrI}
 Z_L^{1/n}(1-\eta_L)
 \leq\AT(n,p)\|\nabla^\Sigma f\|_{L^p(\Sigma)}
 +\frac{c_{n,p}^{-1/n}L}{n}
       \int_\Sigma |H|f^\beta\dd\vol_\Sigma.
\end{equation}
\end{proposition}

\begin{proof}
Apply Proposition~\ref{OTProp} with $F=f^{p^*}$ and
target density $G_L=Z_L^{-1}\ind_{B_L^N}\mathcal{G}$. Since $G_L>0$ almost
everywhere in $B_L^N$, both \eqref{AreaB} and
\eqref{Ball} hold. For any $v\in A_x$, it holds that 
\[
 0\leq\tr Q_{x,v}
 =\Delta^\Sigma_{\ac}u-\langle H,v\rangle
 \leq\Delta^\Sigma_{\ac}u+L|H|.
\]
Thus \eqref{Det} yields $0\leq\det Q_{x,v}\leq n^{-n}(\Delta^\Sigma_{\ac}u+L|H|)^n$. 
With the choice of $G_L$, \eqref{MA} becomes 
\begin{equation}\label{MA1}
	Z_Lf^{p^*}(x)
	=\int_{A_x}\mathcal{G}(\Phi(x,v))\det Q_{x,v}\dd v, 
\end{equation}
which implies $\Delta^\Sigma_{\ac}u+L|H|>0$ almost
everywhere on $\{f>0\}$. At these points, define
\[
 h_x(v)=
 \begin{cases}
 \displaystyle\frac{n^n\det Q_{x,v}}
 {(\Delta^\Sigma_{\ac}u(x)+L|H(x)|)^n},&v\in A_x,\\[5pt]
 0,&v\notin A_x,
 \end{cases}
\]
and
\[
 r(x)=\frac{1}{g_{n,p}(\nabla^\Sigma u(x))}
       \int_{T_x^\perp\Sigma}h_x(v)\mathcal{G}(\nabla^\Sigma u(x)+v)\dd v.
\]
Then $0\leq h_x\leq1$ and $0<r\leq1$. We extend $r$ to all of $\Sigma$ by setting $r=1$ elsewhere. The Monge--Amp\`ere identity \eqref{MA1} then becomes
\begin{equation}\label{Root}
 Z_L^{1/n}f^{p^*/n}(1+|\nabla^\Sigma u|^q)r^{-1/n}
 =\frac{c_{n,p}^{-1/n}}n(\Delta^\Sigma_{\ac}u+L|H|)
\end{equation}
almost everywhere on $\{f>0\}$. Set
\begin{align*}
 J=\int_\Sigma f^{p^*}|\nabla^\Sigma u|^q\dd\vol_\Sigma,\qquad
 \delta=\int_\Sigma f^{p^*}(1+|\nabla^\Sigma u|^q)
                    (r^{-1/n}-1)\dd\vol_\Sigma.
\end{align*}
Then $0\leq J\leq L^q$. Since $\beta+p^*/n=p^*$, multiplying
\eqref{Root} by $f^\beta$ and integrating, we obtain 
\begin{equation}\label{Int}
 Z_L^{1/n}(1+J+\delta)
 =\frac{c_{n,p}^{-1/n}}n
   \int_\Sigma f^\beta(\Delta^\Sigma_{\ac}u+L|H|)\dd\vol_\Sigma.
\end{equation}
Since $\Delta^\Sigma_{\ac}u\in L^1_{\mathrm{loc}}(\Sigma)$,
$H$ is smooth, and $f$ is compactly supported, the right-hand
side is finite, which implies $\delta<\infty$. 

Since $\beta>1$, we have $f^\beta\in C_c^1(\Sigma)$. 
By approximation in $C^1$ with nonnegative smooth functions
supported in a fixed compact set, we may use $f^\beta$ as a test
function in
$\Delta^\Sigma_{\ac}u\leq\Delta^\Sigma_{\mathcal D}u$. By integration by parts and H\"older's inequality, using 
$q(\beta-1)=p^*$, we have  
\begin{align*}
 \int_\Sigma f^\beta(\Delta^\Sigma_{\ac}u+L|H|)\dd\vol_\Sigma
 &\leq-\beta\int_\Sigma f^{\beta-1}
       \langle\nabla^\Sigma f,\nabla^\Sigma u\rangle\dd\vol_\Sigma
       +L\int_\Sigma|H|f^\beta\dd\vol_\Sigma\\
 &\leq\beta\|\nabla^\Sigma f\|_{L^p(\Sigma)}J^{1/q}
       +L\int_\Sigma|H|f^\beta\dd\vol_\Sigma.
\end{align*}
Combining this with \eqref{Int}, we obtain
\begin{equation}\label{En}
 Z_L^{1/n}(1+J+\delta)
 \leq\frac\beta n c_{n,p}^{-1/n}
       \|\nabla^\Sigma f\|_{L^p(\Sigma)}J^{1/q}
       +\frac{c_{n,p}^{-1/n}L}{n}
        \int_\Sigma|H|f^\beta\dd\vol_\Sigma.
\end{equation}

We next relate $J$ to $\delta$. Applying Lemma~\ref{Fib} with $V=T_x\Sigma$, $z=\nabla^\Sigma u(x)$ and $h=h_x$, and by \eqref{MA1}, it holds that 
\[
 \frac{\displaystyle\int_{A_x}\mathcal{K}(\Phi(x,v))\det Q_{x,v}\dd v}
      {Z_Lf^{p^*}(x)}
 \leq(1+|\nabla^\Sigma u(x)|^q)
      \left(\frac n{n-1}r(x)^{-1/n}-1\right)
\]
for almost every $x\in\{f>0\}$. Multiplying by $f^{p^*}$ and
using \eqref{AreaB}, we obtain
\begin{align*}
 \frac{M_L}{Z_L}=\frac1{Z_L}\int_\Sigma\int_{A_x}
       \mathcal{K}(\Phi(x,v))\det Q_{x,v}\dd v\dd\vol_\Sigma(x)\leq\frac{1+J+n\delta}{n-1}.
\end{align*}
By \eqref{Err}, this implies
$\JE-J\leq n(\delta+\eta_L)$. Then by Lemma~\ref{Sc}, we obtain 
\[
 \frac{J^{1/q}}{1+J+\delta+\eta_L}
 \leq\frac{\JE^{1/q}}{1+\JE}.
\]
Add $Z_L^{1/n}\eta_L$ to both sides of
\eqref{En} and divide by
$1+J+\delta+\eta_L\geq1$. Since $\beta=1+\JE$, \eqref{ATJ}
yields
\[
 Z_L^{1/n}
 \leq\AT(n,p)\|\nabla^\Sigma f\|_{L^p(\Sigma)}
       +\frac{c_{n,p}^{-1/n}L}{n}
        \int_\Sigma|H|f^\beta\dd\vol_\Sigma
       +Z_L^{1/n}\eta_L, 
\]
which is the thesis. 
\end{proof}

\begin{remark}
For an embedded minimal submanifold, letting $L\to\infty$ in
\eqref{TrI} proves \eqref{SobI}.
The mean curvature term is retained since the approximating
embeddings used below need not be minimal.
\end{remark}

\subsection{Approximation of minimal immersions}

\begin{lemma}\label{Emb}
Let $X:(\Sigma,g)\to\R^{n+m}$ be a complete
$n$-dimensional isometric immersion without boundary. Then there exists a
smooth proper embedding $\Psi:\Sigma\to\R^k$, with $k=2n+1$, such that
for every $\varepsilon>0$,
\begin{equation*}
 X_\varepsilon=(X,\varepsilon\Psi):\Sigma\to\R^{n+m+k}
\end{equation*}
is a smooth proper embedding. Its induced metric is complete and
satisfies
\begin{equation}\label{Met}
 g_\varepsilon=g+\varepsilon^2\Psi^*g_{\Euc}
 \longrightarrow g\quad\mathrm{in }\,\,C^\infty_{\mathrm{loc}}.
\end{equation}
Writing $H_X=\Delta_gX$ and
$H_\varepsilon=\Delta_{g_\varepsilon}X_\varepsilon$
for the unnormalized mean curvature vectors, we have
\begin{equation}\label{Hlim}
 H_\varepsilon\longrightarrow(H_X,0)
 \quad\text{in }C^\infty_{\mathrm{loc}}.
\end{equation}
If $X$ is minimal, then for every compact $K\subset \Sigma$, there exist 
$C_K,\varepsilon_K>0$ such that
\begin{equation}\label{Hrate}
 \sup_K|H_\varepsilon|\leq C_K\varepsilon,
 \qquad \forall\, \varepsilon\in (0,\varepsilon_K).
\end{equation}
\end{lemma}

\begin{proof}
The proper version of Whitney's embedding theorem gives
$\Psi:\Sigma\to\R^{2n+1}$; see~\cite[Theorem~6.15]{Lee}.
The map $X_\varepsilon$ is an injective immersion since $\Psi$
is an embedding. It is proper, since the inverse image of a compact set is closed and is contained in the inverse image under
$\Psi$ of a compact set. Hence $X_\varepsilon$ is an embedding. 

Since the two Euclidean factors are orthogonal and $X$ is
isometric, the induced metric is
$g_\varepsilon=g+\varepsilon^2\Psi^*g_{\Euc}$. 
This proves the identity in \eqref{Met}. In particular,
$g_\varepsilon\geq g$. A $g_\varepsilon$-Cauchy sequence is therefore $g$-Cauchy
and converges to some $s\in\Sigma$. Since convergence on a
smooth manifold is independent of the choice of Riemannian
metric, it also converges to $s$ with respect to
$g_\varepsilon$, which implies $g_\varepsilon$ is complete. 

On every compact set, $g_\varepsilon-g=O(\varepsilon^2)$ in each
$C^r$-norm. Since $H_\varepsilon=\Delta_{g_\varepsilon}X_\varepsilon$,
\begin{equation*}
 H_\varepsilon
 =\bigl(\Delta_{g_\varepsilon}X,
        \varepsilon\Delta_{g_\varepsilon}\Psi\bigr)
 =\bigl(H_X+O_K(\varepsilon^2),
        \varepsilon\Delta_\Sigma\Psi+O_K(\varepsilon^3)\bigr), 
\end{equation*}
which proves \eqref{Hlim} and
\eqref{Hrate}.
\end{proof}

\begin{proof}[Proof of Theorem~\ref{Sob}: the inequality part]
It suffices to work on one connected component of $\Sigma$: a
compact support meets only finitely many components, and their
estimates combine by the inclusion $\ell^p\subset\ell^{p^*}$.
Take a nonnegative $f\in C_c^\infty(\Sigma)$ with
$\int_\Sigma f^{p^*}\dd\vol_\Sigma=1$.
Let $K=\supp f$, and use the embeddings in
Lemma~\ref{Emb}. Put
\[
 N=n+m+k,\qquad \Sigma_\varepsilon=X_\varepsilon(\Sigma)\subset\R^N.
\]
Fix $L>0$, and let $Z_L$ and $\eta_L$ be defined by
\eqref{ZLML} and \eqref{Err}, respectively, for this fixed
dimension $N$. These quantities depend only on $n,p,N$ and $L$
and are therefore independent of $\varepsilon$. 

Let
\begin{equation*}
 A_\varepsilon=\left(\int_\Sigma f^{p^*}\dd\vol_{g_\varepsilon}\right)^{1/p^*},
 \qquad f_\varepsilon(X_\varepsilon(x))=\frac{f(x)}{A_\varepsilon}.
\end{equation*}
Then $f_\varepsilon\in C_c^\infty(\Sigma_\varepsilon)$ is
nonnegative and satisfies
$\|f_\varepsilon\|_{L^{p^*}(\Sigma_\varepsilon)}=1$. Applying Proposition~\ref{Tr} to
$(\Sigma_\varepsilon,f_\varepsilon)$ and pulling the resulting
inequality back to $\Sigma$ via $X_\varepsilon$, we obtain 
\begin{align} \label{ImmI}
 Z_L^{1/n}(1-\eta_L)
 \leq\frac{\AT(n,p)}{A_\varepsilon}
\left(\int_\Sigma|\nabla_{g_\varepsilon}f|^p\dd\vol_{g_\varepsilon}\right)^{1/p}+\frac{c_{n,p}^{-1/n}L}{nA_\varepsilon^\beta}
\int_\Sigma|H_\varepsilon|f^\beta\dd\vol_{g_\varepsilon}.
\end{align}
Smooth convergence of the metrics on $K$ implies
\begin{equation*}
 A_\varepsilon\longrightarrow1,\qquad
 \int_\Sigma|\nabla_{g_\varepsilon}f|^p\dd\vol_{g_\varepsilon}
 \longrightarrow\int_\Sigma|\nabla^\Sigma f|^p\dd\vol_\Sigma.
\end{equation*}
Since $X$ is minimal, \eqref{Hrate} also yields 
\[
 A_\varepsilon^{-\beta}
 \int_\Sigma|H_\varepsilon|f^\beta\dd\vol_{g_\varepsilon}
 \leq\frac{C_K\varepsilon}{A_\varepsilon^\beta}
       \int_K f^\beta\dd\vol_{g_\varepsilon}
 \longrightarrow0.
\]
For fixed $L$, letting $\varepsilon\downarrow0$ in \eqref{ImmI}, we obtain 
\[
 Z_L^{1/n}(1-\eta_L)
 \leq\AT(n,p)\|\nabla^\Sigma f\|_{L^p(\Sigma,g)}.
\]
Letting $L\to\infty$ and using \eqref{Lim}, we obtain
\eqref{SobI} for normalized nonnegative $f$.
By homogeneity, the same inequality holds for every nonnegative
$f\in C_c^\infty(\Sigma)$.

Finally, for an arbitrary real-valued $f\in C_c^\infty(\Sigma)$,
apply the inequality to
$f_h=\sqrt{f^2+h^2}-h\in C_c^\infty(\Sigma)$, $h>0$.
Since $0\le f_h\le|f|$, $f_h\to|f|$ and
$|\nabla^\Sigma f_h|\le|\nabla^\Sigma f|$,
dominated convergence yields \eqref{SobI} as $h\downarrow0$.
\end{proof}

\subsection{Sharpness and equality}

To prove sharpness, we use the following local argument.

\begin{proposition}\label{UP1}
Let $(\Sigma,g)$ be a nonempty smooth $n$-dimensional Riemannian manifold without
boundary, and let $1<p<n$. Then
\begin{equation*}
 \sup_{\substack{f\in C_c^\infty(\Sigma)\\\|\nabla^\Sigma f\|_{L^p}>0}}
 \frac{\|f\|_{L^{p^*}(\Sigma,g)}}{\|\nabla^\Sigma f\|_{L^p(\Sigma,g)}}
 \geq\AT(n,p).
\end{equation*}
Consequently, the constant in \eqref{SobI} is optimal
for every fixed manifold satisfying Theorem~\ref{Sob}.
\end{proposition}

\begin{proof}
Consider the Euclidean extremal $U(z)=(1+|z|^q)^{-(n-p)/p}$.
It belongs to $L^{p^*}(\R^n)$ and its gradient belongs to
$L^p(\R^n)$. Choose a cut-off function $\chi\in C_c^\infty(\R^n)$ such that 
$0\leq\chi\leq1$, $\chi=1$ on $B_1^n$, and
$\supp\chi\subset B_2^n$. Set $V_R(z)=\chi(z/R)U(z)$.
Then $V_R\to U$ in $L^{p^*}$ and $\nabla V_R\to\nabla U$ in
$L^p$ as $R\to\infty$. Indeed,
\[
\nabla V_R-\nabla U
=(\chi(z/R)-1)\nabla U
+R^{-1}U(z)\nabla\chi(z/R).
\]
The first term tends to zero in $L^p$, while H\"older's inequality for the second term yields 
\[
 R^{-p}\int_{\{R<|z|<2R\}}|U|^p\dd z
 \leq C\left(\int_{\{|z|>R\}}|U|^{p^*}\dd z\right)^{p/p^*}
 \longrightarrow0.
\]
For each fixed $R$, standard mollification approximates $V_R$
in $L^{p^*}$ and its gradient in $L^p$. A diagonal argument therefore gives nonzero nonnegative $\varphi_j\in C_c^\infty(\R^n)$ such that 
\begin{equation*}
	\frac{\|\varphi_j\|_{L^{p^*}(\R^n)}}
	{\|\nabla\varphi_j\|_{L^p(\R^n)}}
	\longrightarrow
	\frac{\|U\|_{L^{p^*}(\R^n)}}
	{\|\nabla U\|_{L^p(\R^n)}}
	=\AT(n,p).
\end{equation*} 

Fix $x_0\in\Sigma$ and choose normal coordinates
$y=\exp_{x_0}^{-1}(x)$ centered at $x_0$. For each fixed $j$
and all sufficiently small $\varepsilon>0$, so that
$\varepsilon\supp\varphi_j$ lies in the coordinate domain, define
\[
f_{\varepsilon,j}(x)
=\varepsilon^{-(n-p)/p}
\varphi_j\bigl(y/\varepsilon\bigr)
\]
in this coordinate neighbourhood and set it equal to zero
elsewhere. Then $f_{\varepsilon,j}\in C_c^\infty(\Sigma)$.
Since $g_{kl}(0)=\delta_{kl}$ and $\supp\varphi_j$ is compact, we have  $g_{kl}(\varepsilon z)\longrightarrow\delta_{kl}$
uniformly for $z\in\supp\varphi_j$. Consequently, the change of
variables $y=\varepsilon z$ yields, as $\varepsilon\downarrow0$,
\begin{align*}
	\int_\Sigma |f_{\varepsilon,j}|^{p^*}\dd\vol_\Sigma
	&\longrightarrow
	\int_{\R^n}|\varphi_j|^{p^*}\dd z,\\
	\int_\Sigma |\nabla^\Sigma f_{\varepsilon,j}|^p\dd\vol_\Sigma
	&\longrightarrow
	\int_{\R^n}|\nabla\varphi_j|^p\dd z.
\end{align*} 
Then letting $j\to\infty$
proves the lower bound. The last assertion follows from
\eqref{SobI}.
\end{proof}

\begin{proposition}\label{Nonatt}
Under the assumptions of Theorem~\ref{Sob}, no nonzero
function in $C_c^\infty(\Sigma)$ attains equality in
\eqref{SobI}.
\end{proposition}

\begin{proof}
By a standard smooth approximation in $W^{1,p}$,
\eqref{SobI} extends to all compactly supported
$W^{1,p}(\Sigma)$ functions. Suppose that a nonzero $f\in C_c^\infty(\Sigma)$ attains equality. By the Sobolev chain rule, replacing $f$ by $|f|$ preserves both its $L^{p^*}$-norm and its $p$-energy. Thus 
$u=|f|/\|f\|_{L^{p^*}(\Sigma,g)}$ is a nonnegative minimizer of the $p$-energy among compactly supported functions $v\in W^{1,p}(\Sigma)$ satisfying $\|v\|_{L^{p^*}(\Sigma)}=1$. Therefore, the first variation of the constrained energy gives 
\begin{equation}\label{EL}
 \int_\Sigma|\nabla^\Sigma u|^{p-2}\langle\nabla^\Sigma u,\nabla^\Sigma \eta\rangle\dd\vol_\Sigma
 =\lambda\int_\Sigma u^{p^*-1}\eta\dd\vol_\Sigma,
 \qquad\forall\,\eta\in C_c^\infty(\Sigma),
\end{equation}
where $\lambda=\int_\Sigma|\nabla^\Sigma u|^p\dd\vol_\Sigma=\AT(n,p)^{-p}>0$.
Choose $\eta$ equal to $1$ on a neighbourhood of $\supp u$.
The left side of \eqref{EL} vanishes, while the right
side equals $\lambda\int_\Sigma u^{p^*-1}\dd\vol_\Sigma>0$, a contradiction.
\end{proof}

\begin{remark}
By completion, \eqref{SobI} extends to the homogeneous
Sobolev space $\mathcal D^{1,p}(\Sigma)$ generated by $C_c^\infty(\Sigma)$
under the norm $\|\nabla^\Sigma f\|_{L^p(\Sigma,g)}$. On an affine $n$-plane,
the classical nonzero extremals are the Aubin--Talenti functions 
\[
 f(x)=C\bigl(a+|x-x_0|^q\bigr)^{-(n-p)/p},
\]
where $C\ne0$, $a>0$, and $x_0$ belongs to the plane; see
\cite{Aubin,Talenti}.  They belong to $\mathcal D^{1,p}(\R^n)$
but are not compactly supported. Proposition~\ref{Nonatt} only rules out compactly
supported smooth extremals; it does not address extremals in
$\mathcal D^{1,p}(\Sigma)$ for a general minimal immersion. 
\end{remark}

\section{Applications}\label{App}

Throughout this section, unqualified norms and integrals are taken
with respect to the induced Riemannian volume, and
$\|\cdot\|_s=\|\cdot\|_{L^s}$.
Every connected component of $\Sigma$ is noncompact: otherwise the
coordinate functions of the minimal immersion would be harmonic
on a compact manifold and hence constant.

\subsection{Sharp Gagliardo--Nirenberg inequalities}

We use the dimension-reduction argument described in
Bakry, Gentil and Ledoux~\cite{BGL}, in the form developed
further by Nguyen~\cite{Nguyen}, and apply it to the product of
the minimal immersion with a Euclidean space.

\begin{proof}[Proof of Corollary~\ref{GN}]
The map $(x,y)\mapsto(X(x),y)$ is a complete minimal immersion
of $\Sigma\times\R^k$ into $\R^{n+m+k}$, with the product metric.
Since $d=n+k\geq3$, Theorem~\ref{Sob}, applied in dimension $d$
with $p=2$, yields the sharp Sobolev inequality on this product.
By approximation and replacement of $f$ by $|f|$, it suffices to
consider a nonzero nonnegative compactly supported Lipschitz
function $f$. Put $a=1/(t-1)=(2n+k-4)/4$ and $b=2a+1$. For
$\lambda>0$, define
\[
 F_\lambda(x,y)
 =\bigl(\lambda f(x)^{-1/a}+|y|^2\bigr)^{-(d-2)/2}
 \quad\hbox{on }\{f>0\},
\]
and set $F_\lambda=0$ on $\{f=0\}$. Equivalently,
\[
 F_\lambda
 =f^{(d-2)/(2a)}
       (\lambda+|y|^2f^{1/a})^{-(d-2)/2}.
\]
Since $(d-2)/(2a)>1$, both $F_\lambda$ and its first derivatives tend to
zero as $f\downarrow0$. Hence its extension
by zero to $\{f=0\}$ is locally Lipschitz, and its weak gradient
equals the classical gradient on $\{f>0\}$ and vanishes almost
everywhere on $\{f=0\}$. Put $c=\pi^{k/2}\Gamma(n+k/2)/\Gamma(d)$. Note that for any $A>0$, 
\[
\int_{\R^k}(A+|y|^2)^{-d}\dd y
=cA^{-n-k/2},
\qquad
\int_{\R^k}|y|^2(A+|y|^2)^{-d}\dd y
=\frac{kc}{2b}A^{1-n-k/2}.
\]
Therefore, integrating in $y$ and using the definitions of $a$, $b$, and $t$, we obtain 
\[
 \int_{\Sigma\times\R^k}F_\lambda^{2d/(d-2)}\dd\vol_\Sigma(x)\dd y
 =c\lambda^{-n-k/2}\|f\|_{2t}^{2t}
\]
and
\[
 \int_{\Sigma\times\R^k}|\nabla F_\lambda|^2\dd\vol_\Sigma(x)\dd y
 =c(d-2)^2\left[
 \frac{\lambda^{2-n-k/2}}{4a^2}\|\nabla^\Sigma f\|_2^2
 +\frac{k\lambda^{1-n-k/2}}{2b}\|f\|_{t+1}^{t+1}\right].
\]

Since $n+k/2-2>0$, integrating in $y$ yields 
$F_\lambda\in L^2(\Sigma\times\R^k)$. By cutting off $F_\lambda$ in the $y$-variables and then using
smooth approximation, as in the proof of
Proposition~\ref{UP1}, we may pass to the limit in
Theorem~\ref{Sob} and obtain 
\begin{equation}\label{GNOpt}
 \|f\|_{2t}^{2t(d-2)/d}
 \leq\AT(d,2)^2c^{2/d}(d-2)^2
 \left[\frac{\lambda^{k/d}}{4a^2}\|\nabla^\Sigma f\|_2^2
       +\frac{k\lambda^{-n/d}}{2b}\|f\|_{t+1}^{t+1}\right].
\end{equation}
Since $\|\nabla^\Sigma f\|_2>0$, the right-hand side is minimized at $\lambda=\frac{2na^2}{b}
         \frac{\|f\|_{t+1}^{t+1}}{\|\nabla^\Sigma f\|_2^2}.$
Combining this and taking the power $d/[2t(d-2)]$
gives an inequality of the form \eqref{GNI}, with
\begin{equation}\label{GNSc}
 \frac1{2t}=\theta\left(\frac12-\frac1n\right)
                  +\frac{1-\theta}{t+1}.
\end{equation}

To identify its constant, take $\Sigma=\R^n$ and
$U(x)=(1+|x|^2)^{-a}$. For $\lambda=1$, the corresponding
$F_1(x,y)=(1+|x|^2+|y|^2)^{-(d-2)/2}$ is a Sobolev
extremizer on $\R^d$. Also $\frac{\int_{\R^n}|\nabla U|^2\dd x}
      {\int_{\R^n}U^{t+1} \dd x}=\frac{2na^2}{b},$
so $\lambda=1$ is the minimizing value in
\eqref{GNOpt}. Thus the constant obtained by optimization is the quotient
of the three norms of $U$. This comparison takes place in the homogeneous Sobolev space
$\mathcal D^{1,2}(\R^d)$. Finally,
\[
 \int_{\R^n} U^{t+1}\dd x=\pi^{n/2}\frac{\Gamma(b-n/2)}{\Gamma(b)},
 \qquad
 \frac{\int_{\R^n} U^{2t}\dd x}{\int_{\R^n} U^{t+1}\dd x}=\frac{b-n/2}{b},
\]
which gives the stated expression for $G_{n,t}$.

For optimality on a fixed manifold, let $C_\Sigma$ be any constant
for which \eqref{GNI} holds. Choose a cut-off function $\chi\in C_c^\infty(\R^n)$ such that 
$0\leq\chi\leq1$, $\chi=1$ on $B_1^n$, and
$\supp\chi\subset B_2^n$. Set $V_R(x)=\chi(x/R)U(x)$. Then $V_R\to U$ in $L^{t+1}\cap L^{2t}$. Since $4a=2n+k-4$, 
\[
 R^{-2}\int_{\{R<|x|<2R\}}U^2\dd x
 \leq C R^{n-2-4a}=C R^{2-d}\longrightarrow0.
\]
Hence $\nabla V_R\to\nabla U$ in $L^2$. For each fixed $R$, applying the normal coordinate rescaling argument as in the proof of
Proposition~\ref{UP1} to $V_R$, and using \eqref{GNSc}, we obtain 
\[
C_\Sigma\geq
\frac{\|V_R\|_{2t}}
{\|\nabla V_R\|_2^\theta\|V_R\|_{t+1}^{1-\theta}}.
\]
Letting $R\to\infty$ yields $C_\Sigma\geq G_{n,t}$.
Therefore $G_{n,t}$ is optimal on every fixed manifold. 
\end{proof}

\begin{remark}
	By completion, \eqref{GNI} extends to the completion of
	$C_c^\infty(\Sigma)$ with respect to the norm $\|\nabla^\Sigma f\|_{L^2(\Sigma)}
	+\|f\|_{L^{t+1}(\Sigma)}$. 
	On an affine $n$-plane $P$, the classical nonzero extremals are
	\[
	f(x)=C\bigl(\lambda+|x-x_0|^2\bigr)^{-1/(t-1)},
	\qquad
	C\ne0,\quad \lambda>0,\quad x_0\in P;
	\]
	see~\cite{DPD}. They belong to this completion but are not
	compactly supported. The cut-off argument above shows that the
	inequality on $C_c^\infty(P)$ has the same best constant.
\end{remark} 

\subsection{Beckmann duality and Green kernels}

By Theorem~\ref{Sob}, the inclusion
$\mathcal D^{1,p}(\Sigma)\hookrightarrow L^{p^*}(\Sigma)$ has norm
$\AT(n,p)$.

\begin{proof}[Proof of Corollary~\ref{Dual}]
Define
\[
 N(h)=\sup\left\{\int_\Sigma h\varphi\dd\vol_\Sigma:
 \varphi\in C_c^\infty(\Sigma),\ \|\nabla^\Sigma \varphi\|_p\leq1\right\}.
\]
Since $1/p^*+1/r=1$, by Theorem~\ref{Sob} and H\"older's inequality, we have  $N(h)\leq\AT(n,p)\|h\|_r$.
The functional
$\nabla^\Sigma \varphi\mapsto\int_\Sigma h\varphi\dd\vol_\Sigma$ on the gradient
subspace of $L^p(T\Sigma)$ has norm $N(h)$. By the Hahn--Banach theorem~\cite[Theorem~1.1]{Brezis}
and $L^p$ duality, there exists $Z\in L^q(T\Sigma)$ with
$-\operatorname{div}_{\Sigma}Z=h$ and $\|Z\|_q=N(h)$.
Conversely, by H\"older's inequality, every admissible field has norm at least $N(h)$.
Consequently,
\begin{equation*}
 \mathcal B_q(h)=\frac{N(h)^q}{q}
 \leq\frac{\AT(n,p)^q}{q}\|h\|_r^q.
\end{equation*}
Since the admissible set is convex and the functional
$Z\mapsto q^{-1}\int_\Sigma |Z|^q\dd\vol_\Sigma$ is strictly
convex for $q>1$, the minimizing field is unique up to equality
almost everywhere. Optimizing over scalar multiples of each test function also yields 
the dual formula
\begin{equation*}
 \mathcal B_q(h)=\sup_{\varphi\in C_c^\infty(\Sigma)}
 \left\{\int_\Sigma h\varphi\dd\vol_\Sigma
       -\frac1p\int_\Sigma|\nabla^\Sigma \varphi|^p\dd\vol_\Sigma\right\}.
\end{equation*}

To prove sharpness, fix a probability density
$\eta\in C_c^\infty(\Sigma)$ and vary only $\rho$ in
$h=\rho-\eta$. Choose $x_0\notin\supp\eta$ and normal coordinates about $x_0$.
Let $u\geq0$ be a nonzero smooth compactly supported function on
$\R^n$. By $L^{p^*}$ duality and smooth approximation, one can choose a nonnegative $f\in C_c^\infty(\R^n)$ with
$\|f\|_r=1$ such that $\int_{\R^n}fu\dd z$ is arbitrarily close
to $\|u\|_{p^*}$. 
For small $\varepsilon>0$, define
\[
 u_\varepsilon(\exp_{x_0}(\varepsilon z))
 =\varepsilon^{-(n-p)/p}u(z),\qquad
 f_\varepsilon(\exp_{x_0}(\varepsilon z))
 =\varepsilon^{-n/r}f(z),
\]
and extend both functions by zero. Their supports are disjoint
from $\supp\eta$. Put
$t_\varepsilon=\int_\Sigma f_\varepsilon\dd\vol_\Sigma$ and
$h_\varepsilon=f_\varepsilon-t_\varepsilon\eta$. By the change of variables
$x=\exp_{x_0}(\varepsilon z)$ in normal coordinates, and using
$n-n/r=(n-p)/p$, we obtain
\[
t_\varepsilon
=\varepsilon^{(n-p)/p}
\left(\int_{\R^n}f\dd z+o(1)\right),\qquad
\|f_\varepsilon\|_r\longrightarrow1,\qquad
\|\nabla^\Sigma u_\varepsilon\|_p
\longrightarrow\|\nabla u\|_p. 
\]
The disjointness of the supports then yields $\|h_\varepsilon\|_r^r
=\|f_\varepsilon\|_r^r+t_\varepsilon^r\|\eta\|_r^r
\longrightarrow1$. Moreover, 
\[
\liminf_{\varepsilon\downarrow0}N(h_\varepsilon)
\geq\lim_{\varepsilon\downarrow0}
\frac{\int_\Sigma h_\varepsilon u_\varepsilon\dd\vol_\Sigma}
{\|\nabla^\Sigma u_\varepsilon\|_p}
=\lim_{\varepsilon\downarrow0}
\frac{\int_\Sigma f_\varepsilon u_\varepsilon\dd\vol_\Sigma}
{\|\nabla^\Sigma u_\varepsilon\|_p}
=\frac{\int_{\R^n}fu\dd z}{\|\nabla u\|_p}.
\] 
Set $\rho_\varepsilon=f_\varepsilon/t_\varepsilon$.
Homogeneity shows that
\[
 \frac{\mathcal B_q(\rho_\varepsilon-\eta)}
      {\|\rho_\varepsilon-\eta\|_r^q}
 =\frac{\mathcal B_q(h_\varepsilon)}{\|h_\varepsilon\|_r^q}.
\]
Taking $u$ along the Euclidean Sobolev extremizing sequence as in the proof of Proposition~\ref{UP1}, and choosing $f$ as above for each $u$, proves the sharpness of \eqref{BI} with $\eta$ fixed. 

We now take $p=2$ and $r=2n/(n+2)$. By the Riesz representation theorem~\cite[Theorem~5.5]{Brezis},
for each $h\in L^r(\Sigma)$, there exists a unique
$Th\in\mathcal D^{1,2}(\Sigma)$ such that 
\begin{equation}\label{Riesz}
 \int_\Sigma\langle\nabla^\Sigma Th,\nabla^\Sigma u\rangle\dd\vol_\Sigma
 =\int_\Sigma hu\dd\vol_\Sigma,
 \qquad u\in\mathcal D^{1,2}(\Sigma).
\end{equation}
Thus $N(h)=\|\nabla^\Sigma Th\|_2$ and
$2\mathcal B_2(h)=\int_\Sigma hTh\dd\vol_\Sigma$. It follows that
\[
 \left|\int_\Sigma fTh\dd\vol_\Sigma\right|
 =\left|\int_\Sigma\langle\nabla^\Sigma Tf,\nabla^\Sigma Th\rangle\dd\vol_\Sigma\right|
 \leq\AT(n,2)^2\|f\|_r\|h\|_r.
\]
We identify $T$ with the minimal Green operator. The Sobolev
inequality yields 
\begin{equation*}
 \operatorname{cap}_2(K)
 \geq \AT(n,2)^{-2}\vol_\Sigma(K)^{(n-2)/n}>0 
\end{equation*}
for every compact set $K$ of positive volume. Here
$\operatorname{cap}_2(K)$ is the infimum of
$\int_\Sigma|\nabla^\Sigma \varphi|^2\dd\vol_\Sigma$ over
$\varphi\in C_c^\infty(\Sigma)$ with $\varphi\geq1$ near $K$.
Each component is therefore nonparabolic and has a finite minimal
Green kernel off the diagonal; see~\cite[Theorem~5.1]{Grigoryan}.
On different components we set $G_\Sigma=0$.

Let $\Omega_j$ be an increasing smooth relatively compact
exhaustion, and let $T_j$ be the Dirichlet Green operator on
$\Omega_j$. For $0\leq h\in C_c^\infty(\Sigma)$, extend $T_jh$ by
zero. In $\mathcal D^{1,2}(\Sigma)$, it is the orthogonal projection
of $Th$ onto $H_0^1(\Omega_j)$. The union of these subspaces is
dense, so $T_jh\to Th$ in $\mathcal D^{1,2}(\Sigma)$ and in
$L^{2n/(n-2)}(\Sigma)$. On the other hand, the Dirichlet Green kernels increase
to $G_\Sigma$; see~\cite[Section~4.2]{Grigoryan}.
Monotone convergence therefore identifies the limit obtained
above as
\[
Th(x)=\int_\Sigma G_\Sigma(x,y)h(y)\dd\vol_\Sigma(y)
\quad\text{for almost every }x\in\Sigma.
\]

To extend this representation to $L^r(\Sigma)$, we first
approximate nonnegative $f,h\in L^r(\Sigma)$ by nonnegative
functions in $C_c^\infty(\Sigma)$.
Fatou's lemma then gives the same bilinear bound for the
Green integral, which consequently defines a bounded operator
from $L^r(\Sigma)$ to $L^{2n/(n-2)}(\Sigma)$.
Since this operator agrees with $T$ on $C_c^\infty(\Sigma)$,
the two operators agree on all of $L^r(\Sigma)$ by density.
Finally, applying the bound to $|f|$ and $|h|$ shows that
the double integral in \eqref{HLS} is absolutely convergent
for arbitrary $f,h\in L^r(\Sigma)$ and yields \eqref{HLS}.

To prove sharpness, suppose that the quadratic estimate holds
with a constant $C$ for all nonnegative $h\in L^r(\Sigma)$.
For nonnegative $u\in C_c^\infty(\Sigma)$ and $h\in L^r(\Sigma)$,
\eqref{Riesz} and the Cauchy--Schwarz inequality give
\[
\int_\Sigma hu\dd\vol_\Sigma
\leq
\left(\int_\Sigma hTh\dd\vol_\Sigma\right)^{1/2}
\|\nabla^\Sigma u\|_2
\leq C^{1/2}\|h\|_r\|\nabla^\Sigma u\|_2.
\]
Taking the supremum over $h\geq0$ with $\|h\|_r\leq1$, we obtain
\[
\|u\|_{2n/(n-2)}
\leq C^{1/2}\|\nabla^\Sigma u\|_2.
\]
Combining this with the sharpness assertion in Theorem~\ref{Sob},
we conclude that $C\geq\AT(n,2)^2$, which proves the sharpness
of \eqref{HLS}.
\end{proof}

\begin{remark}
The divergence constraint in \eqref{Beck} is tested
against compactly supported functions. It does not impose zero
flux at infinity. The estimate is therefore stated for arbitrary
$h\in L^r(\Sigma)$, without a zero-mass assumption. The fixed-target
sharpness assertion concerns the same variational problem.
\end{remark} 

% \appendix
% \section{The Aubin--Talenti constant}

% Polar coordinates and the substitution $t=r^q$ give
% \[
%  c_{n,p}=n\omega_n\int_0^\infty r^{n-1}(1+r^q)^{-n}\dd r
%  =\frac{n\omega_n}{q}B(n/q,n/p),
% \]
% which is \eqref{c}. For
% $U(z)=(1+|z|^q)^{-(n-p)/p}$,
% \[
%  U^{p^*}=(1+|z|^q)^{-n},\qquad
%  |\nabla U|^p=\left(\frac{n-p}{p-1}\right)^p
%                |z|^q(1+|z|^q)^{-n}.
% \]
% Using \eqref{J},
% \[
%  \frac{\|U\|_{L^{p^*}(\R^n)}}{\|\nabla U\|_{L^p(\R^n)}}
%  =\frac{p-1}{n-p}c_{n,p}^{-1/n}\JE^{-1/p}
%  =\frac1n c_{n,p}^{-1/n}\JE^{1/q}.
% \]
% Finally, $\omega_n=\pi^{n/2}/\Gamma(n/2+1)$ and
% \[
%  \frac1n\JE^{1/q}
%  =n^{-1/p}\left(\frac{p-1}{n-p}\right)^{1/q}
% \]
% show that \eqref{ATJ} agrees with \eqref{AT}.


\begin{thebibliography}{99}

\bibitem{AldrigoBalogh}
P. Aldrigo and Z. M. Balogh,
\emph{P\'olya--Szeg\H{o} inequalities on submanifolds with small
total mean curvature},
J. Geom. Anal. \textbf{35} (2025), Paper No.~394.

\bibitem{Allard}
W. K. Allard,
\emph{On the first variation of a varifold},
Ann. of Math. (2) \textbf{95} (1972), 417--491.

\bibitem{Aubin}
T. Aubin,
\emph{Probl\`emes isop\'erim\'etriques et espaces de Sobolev},
J. Differential Geom. \textbf{11} (1976), 573--598.

\bibitem{BDKLog}
Z. M. Balogh, S. Don and A. Krist\'aly,
\emph{Sharp weighted log-Sobolev inequalities:
characterization of equality cases and applications},
Trans. Amer. Math. Soc. \textbf{377} (2024), no.~7, 5129--5163.

\bibitem{BGL}
D. Bakry, I. Gentil and M. Ledoux,
\emph{Analysis and Geometry of Markov Diffusion Operators},
Grundlehren der mathematischen Wissenschaften, vol.~348,
Springer, Cham, 2014.

\bibitem{BK}
Z. M. Balogh and A. Krist\'aly,
\emph{Logarithmic-Sobolev inequalities on non-compact Euclidean
submanifolds: sharpness and rigidity},
J. Eur. Math. Soc. (2026), published online first.

\bibitem{BKM}
Z. M. Balogh, A. Krist\'aly and \'A. Mester,
\emph{$L^p$-Sobolev inequalities on minimal submanifolds},
J. Lond. Math. Soc. (2) \textbf{113} (2026), no.~4, e70521.

\bibitem{Beckmann}
M. Beckmann,
\emph{A continuous model of transportation},
Econometrica \textbf{20} (1952), no.~4, 643--660.

\bibitem{BDGM}
E. Bombieri, E. De Giorgi and M. Miranda,
\emph{Una maggiorazione a priori relativa alle ipersuperfici
minimali non parametriche},
Arch. Rational Mech. Anal. \textbf{32} (1969), 255--267.

\bibitem{BrascoPetrache}
L. Brasco and M. Petrache,
\emph{A continuous model of transportation revisited},
J. Math. Sci. \textbf{196} (2014), no.~2, 119--137.

\bibitem{BrendleIsoperimetric}
S. Brendle,
\emph{The isoperimetric inequality for a minimal submanifold in Euclidean space},
J. Amer. Math. Soc. \textbf{34} (2021), 595--603.

\bibitem{BrendleLog}
S. Brendle,
\emph{The logarithmic Sobolev inequality for a submanifold in Euclidean space},
Comm. Pure Appl. Math. \textbf{75} (2022), 449--454.

\bibitem{BETransport}
S. Brendle and M. Eichmair,
\emph{Proof of the Michael--Simon--Sobolev inequality using optimal transport},
J. Reine Angew. Math. \textbf{804} (2023), 1--10.

\bibitem{BrendleSurvey}
S. Brendle,
\emph{Minimal hypersurfaces and geometric inequalities},
Ann. Fac. Sci. Toulouse Math. (6) \textbf{32} (2023),
no.~1, 179--201.

\bibitem{BENotices}
S. Brendle and M. Eichmair,
\emph{The isoperimetric inequality},
Notices Amer. Math. Soc. \textbf{71} (2024), no.~6, 708--713.

\bibitem{Brenier}
Y. Brenier,
\emph{Polar factorization and monotone rearrangement of
vector-valued functions},
Comm. Pure Appl. Math. \textbf{44} (1991), no.~4, 375--417.

\bibitem{Brezis}
H. Brezis,
\emph{Functional Analysis, Sobolev Spaces and Partial Differential
Equations},
Universitext, Springer, New York, 2011.

\bibitem{CabreMiraglio}
X. Cabr\'e and P. Miraglio,
\emph{Universal Hardy--Sobolev inequalities on hypersurfaces
of Euclidean space},
Commun. Contemp. Math. \textbf{24} (2022),
no.~5, Paper No.~2150063.

\bibitem{Castillon}
P. Castillon,
\emph{Submanifolds, isoperimetric inequalities and optimal transportation},
J. Funct. Anal. \textbf{259} (2010), no.~1, 79--103.

\bibitem{CorderoGaussian}
D. Cordero-Erausquin,
\emph{Some applications of mass transport to Gaussian-type
inequalities},
Arch. Ration. Mech. Anal. \textbf{161} (2002), no.~3, 257--269.

\bibitem{CNV}
D. Cordero-Erausquin, B. Nazaret and C. Villani,
\emph{A mass-transportation approach to sharp Sobolev and
Gagliardo--Nirenberg inequalities},
Adv. Math. \textbf{182} (2004), no.~2, 307--332.

\bibitem{DPD}
M. Del Pino and J. Dolbeault,
\emph{Best constants for Gagliardo--Nirenberg inequalities and
applications to nonlinear diffusions},
J. Math. Pures Appl. (9) \textbf{81} (2002), no.~9, 847--875.

\bibitem{Federer}
H. Federer,
\emph{Geometric Measure Theory},
Grundlehren der mathematischen Wissenschaften, vol.~153,
Springer, New York, 1969.

\bibitem{Grigoryan}
A. Grigor'yan,
\emph{Analytic and geometric background of recurrence and
non-explosion of the Brownian motion on Riemannian manifolds},
Bull. Amer. Math. Soc. (N.S.) \textbf{36} (1999), no.~2, 135--249.

\bibitem{HanXu}
B.-X. Han and Z.-F. Xu,
\emph{Optimal Transport and the ABP Method in Higher Codimension},
preprint (2026),
\href{https://arxiv.org/abs/2608.27248}{arXiv:2608.27248}.

\bibitem{HoffmanSpruck}
D. Hoffman and J. Spruck,
\emph{Sobolev and isoperimetric inequalities for Riemannian submanifolds},
Comm. Pure Appl. Math. \textbf{27} (1974), no.~6, 715--727.

\bibitem{Lee}
J. M. Lee,
\emph{Introduction to Smooth Manifolds}, 2nd ed.,
Graduate Texts in Mathematics, vol.~218,
Springer, New York, 2013.

\bibitem{LiebHLS}
E. H. Lieb,
\emph{Sharp constants in the Hardy--Littlewood--Sobolev and related
inequalities},
Ann. of Math. (2) \textbf{118} (1983), 349--374.

\bibitem{McCann}
R. J. McCann,
\emph{Existence and uniqueness of monotone measure-preserving maps},
Duke Math. J. \textbf{80} (1995), no.~2, 309--323.

\bibitem{McCannPass}
R. J. McCann and B. Pass,
\emph{Optimal transportation between unequal dimensions},
Arch. Ration. Mech. Anal. \textbf{238} (2020), no.~3, 1475--1520.

\bibitem{MichaelSimon}
J. H. Michael and L. M. Simon,
\emph{Sobolev and mean-value inequalities on generalized submanifolds
of $R^n$},
Comm. Pure Appl. Math. \textbf{26} (1973), 361--379.

\bibitem{MirandaSobolev}
M. Miranda,
\emph{Diseguaglianze di Sobolev sulle ipersuperfici minimali},
Rend. Sem. Mat. Univ. Padova \textbf{38} (1967), 69--79.

\bibitem{Nguyen}
V. H. Nguyen,
\emph{Sharp weighted Sobolev and Gagliardo--Nirenberg inequalities
on half-spaces via mass transport and consequences},
Proc. Lond. Math. Soc. (3) \textbf{111} (2015), no.~1, 127--148.

\bibitem{Bernstein}
R. L. Schilling, R. Song and Z. Vondra\v cek,
\emph{Bernstein Functions: Theory and Applications}, 2nd ed.,
De Gruyter Studies in Mathematics, vol.~37,
De Gruyter, Berlin, 2012.

\bibitem{Sun}
X.-D. Sun,
\emph{A note on isoperimetric inequality for a minimal submanifold
in Euclidean space},
Differential Geom. Appl. \textbf{102} (2026), Article~102331.

\bibitem{Talenti}
G. Talenti,
\emph{Best constant in Sobolev inequality},
Ann. Mat. Pura Appl. (4) \textbf{110} (1976), 353--372.

\bibitem{Villani}
C. Villani,
\emph{Topics in Optimal Transportation},
Graduate Studies in Mathematics, vol.~58,
American Mathematical Society, Providence, RI, 2003.

\bibitem{Wang}
K.-H. Wang,
\emph{Optimal transport approach to Michael--Simon--Sobolev inequalities
in manifolds with intermediate Ricci curvature lower bounds},
Ann. Global Anal. Geom. \textbf{65} (2024), no.~1, Paper No.~7.
\end{thebibliography}
\end{document}